\documentclass[11pt]{amsart}
\usepackage{oldgerm}
\usepackage{amssymb} 
\usepackage{mathrsfs} 
\usepackage{amsmath}
\usepackage{latexsym}
\usepackage[all]{xy}
\usepackage[dvips]{graphics}
\usepackage{amsthm}
\usepackage{enumerate}
\usepackage{bm}
\usepackage{float}
\usepackage{tikz}
\newcommand{\lt}{{}^\mathrm t\hspace{-0.5mm}}
\newcommand{\ctext}[1]{\raise0.2ex\hbox{\textcircled{\scriptsize{#1}}}}
\newcommand{\pmat}[1]{\begin{pmatrix} #1 \end{pmatrix}}

\theoremstyle{plain} 
\newtheorem{thm}{Theorem}[section] 
\newtheorem{prop}[thm]{Proposition} 
\newtheorem{lem}[thm]{Lemma} 
\newtheorem{cor}[thm]{Corollary} 
\newtheorem*{conj}{Conjecture}
\theoremstyle{definition} 
\newtheorem{ex}[thm]{Example} 
\newtheorem{rem}[thm]{Remark} 
\newtheorem{const}{Construction} 
\theoremstyle{remark} 
\makeatletter
\@namedef{subjclassname@2020}{\textup{2020} Mathematics Subject Classification}
\makeatother

\title[Combinatorics of rational curves on a smooth Hermitian surface]
{Combinatorial structure of 
low degree rational curves on a smooth\\Hermitian surface
}

\author
{Norifumi Ojiro}
\address{
Deta and Mathematical Sciences Course, 
Graduate School of Software
Information Science, Iwate Prefectural University \\
152-52, Sugo, Takizawa, Iwate, 020-0693, Japan
}
\email{norifumi.ojiro@gmail.com}

\date{}

\subjclass[2020]{51E12,51E20,51E25,05E14,05E18,05E30.}
\keywords{strongly regular graph, association scheme,
generalized quadrangle, character table, incidence relation of rational curves in positive characteristic, transitive action by automorphism group of a Hermitian surface.
}

\begin{document}
\maketitle

\begin{abstract}
A smooth Hermitian surface $X$ is a projective surface isomorphic to the Fermat surface of degree $q+1$ in positive characteristic.
We study incidence relations of the rational curves of degree $q+1$ contained in $X$, and show that such curves produce a family of
certain strongly regular graphs and association schemes.
\end{abstract}

\section{Introduction}
Let $k$ be the algebraic closure of a finite field $\mathbb{F}_{q^2}$ for $q$ a power of a prime $p$. A smooth $k$-Hermitian surface is a surface in $\mathbb{P}^3(k)$ isomorphic to the Fermat surface $X_I$ defined by $x_0^{q+1}+x_1^{q+1}+x_2^{q+1}+x_3^{q+1}=0$. Especially, the surface isomorphic over $\mathbb{F}_{q^2}$ to $X_I$ is usually called a Hermitian surface.
 The following proposition is fundamental.
 \begin{prop}[\cite{BC},\cite{C},\cite{Se}]\label{SS}
Let $X$ be a smooth $k$-Hermitian surface and ${\rm Aut}(X)$ the group of projective automorphisms of $X$. Let $X(\mathbb{F}_{q^2})$ be the set of all the $\mathbb{F}_{q^2}$-rational points of $X$ and $\mathcal{L}(X)$ the set of all the lines on $X$. Then ${\rm Aut}(X)$ acts transitively on $X(\mathbb{F}_{q^2}),\mathcal{L}(X)$, and moreover, the numbers
of them are equal to $(q^3+1)(q^2+1),(q^3+1)(q+1)$, respectively.
 \end{prop}
 Since ${\rm Aut}(X)\simeq{\rm PGU}_4(\mathbb{F}_{q^2})$ is very large, $X$ is expected to endow with a variety of combinatorial structures. Indeed, the incidence relation of $X(\mathbb{F}_{q^2})$ and $\mathcal{L}(X)$, and furthermore that of
$\mathcal{L}(X)$ and the set of all tangent planes of $X$ and so on generate certain block designs and strongly regular graphs, see \cite{BC},\cite{C}. On the other hand, an analogue of Proposition \ref{SS} holds for the rational curves of degree $q+1$ contained in $X$.
\begin{prop}[\cite{O},\cite{O2}]\label{O}
Let $X$ and ${\rm Aut}(X)$ be as in Proposition \ref{SS}. Let $\mathcal{R}(X)$ be the set of all the rational curves of degree $q+1$ contained in $X$. Then, if $q\ge3$ the group ${\rm Aut}(X)$ acts transitively on $\mathcal{R}(X)$, and moreover, the number of $\mathcal{R}(X)$ is equal to $q^4(q^3+1)(q^2-1)$, and if $q=2$ then $\mathcal{R}(X)$ is decomposed into infinitely many ${\rm Aut}(X)$-orbits.
\end{prop}
In \cite{Sh1}, for the set $\mathcal{R}$ of the rational normal curves totally tangent to a smooth Hermitian variety $\mathcal{H}$, it was shown that ${\rm Aut}(\mathcal{H})$ acts transitively on $\mathcal{R}$, and moreover, the number of $\mathcal{R}$ is equal to $|{\rm PGU}_{n+1}(\mathbb{F}_{q^2})|/|{\rm PGL}_2(\mathbb{F}_q)|$. Further, using the result to the case of a Hermitian curve, \cite{Sh2} constructed three important examples of strongly regular graphs.

The purpose of this paper is to introduce the combinatorial structure related to $\mathcal{R}(X)$ which will be constructed and examined by
using Propositions \ref{SS}, \ref{O} and other results of \cite{O}, investigating
the incidence relation of $\mathcal{R}(X)$ and observing numerical examples.

The next section is a short preparation. Section 3 constructs strongly regular graphs from $X(\mathbb{F}_{q^2})$ and $\mathcal{R}(X)$, and shows that those graphs form a family associated with the generalized quadrangle ${\rm GQ}(q^2,q)$. Then Section 4 studies association schemes defined by intersection of curves in $\mathcal{R}(X)$. Schurian schemes for $X(\mathbb{F}_{q^2})$, $\mathcal{L}(X)$ and $\mathcal{R}(X)$ are examined in Section 5.

\section{Preliminaries}
Let $\mathcal{M}_{m,n}(A)$ denote the set of all $m$-by-$n$ matrices with entries in a ring $A$, it is abbreviated by $\mathcal{M}_{m}(A)$ if $m=n$.
Without loss of generality, a smooth $k$-Hermitian surface $X$ may be replaced by the Fermat surface $X_I$, because they are projectively equivalent, and so the incidence structures
over those surfaces are constant.
Let us recall a result of \cite{O}.
Let $C_{F_J}$ be the curve in $\mathcal{R}(X_I)$ defined by the image of the map
$$
\mathbb{P}^1(k)\ni\pmat{s\\t}\mapsto F_J\pmat{s^{q+1}\\s^qt\\st^q\\t^{q+1}}\in\mathbb{P}^3(k),
$$
where $F_J$ is an element of $\mathcal{M}_4(\mathbb{F}_{q^2})$
such that
\begin{equation*}
\lt {F_J}{F_J}^{(q)}=\pmat{0&0&0&-1\\0&1&0&0\\0&0&1&0\\-1&0&0&0},
\end{equation*}
and ${F_J}^{(q)}$ is the matrix obtained by taking the $q$-th power of each entry of $F_J$.
Then $\mathcal{R}(X_I)={\rm Aut}(X_I)C_{F_J}$ if $q\ge3$, and so all the curves in $\mathcal{R}(X_I)$ are defined over $\mathbb{F}_{q^2}$ since a complete set of representatives of ${\rm Aut}(X_I)$ can
be chosen from matrices with entries in $\mathbb{F}_{q^2}$. Let $O$ be an ${\rm Aut}(X_I)$-orbit of $\mathcal{R}(X_I)$. Note that $O={\rm Aut}(X_I)C_{F_J}$ for $q\ge3$, whereas for $q=2$ there are infinitely many choices of $O$, see \cite{O2}. 

All examples will be presented in this paper were computed by \texttt{GAP} and \texttt{Macaulay2} when $O={\rm Aut}(X_I)C_{F_J}$.

\section{Strongly regular graphs from rational curves}
\begin{const}\label{con1}
Let $V=X_I(\mathbb{F}_{q^2})$ and $O$ an ${\rm Aut}(X_I)$-orbit of $\mathcal{R}(X_I)$. Let
$$
E=\left\{(\bm{v_1},\bm{v_2})\in V\times V\ \middle\vert\ \bm{v_1},\bm{v_2}\in C\ {\rm for}\ {\rm some}\ C\in O \right\}.
$$
Then $(V,E)$ is a rank $3$ strongly regular graph with vertex number $v=(q^3+1)(q^2+1)$ for each $q$.
\end{const}
\begin{ex}
Let $q=2$. Then $G_2=(V,E)$ is of $(v,k,\lambda,\mu)=(45,32,22,24)$ with ${\rm O}_5(3):2$ as the automorphism group of order $51840$. It is the complement graph of the unique rank 3 strongly regular $(45,12,3,3)$ graph, cf. \cite[\S10.17]{BV}.
\end{ex}
The strongly regular $(45,12,3,3)$ graphs have been completely classified up to isomorphism, there are precisely 78 such graphs, see \cite{CDS}.
\begin{ex}
Let $q=3$. Then $G_3=(V,E)$ is of
$(280,243,210,216)$ with ${\rm PSU}_4(3) : D_8$ as the automorphism group of order $26127360$. It is the complement graph of the unique rank 3 strongly regular $(280,36,8,4)$ graph, cf. \cite[p.292]{BV}.
\end{ex}
From the above examples, it turns out that $G_2,G_3$ are obtained as the complement graphs of the point graphs of
${\rm GQ}(q^2,q)$ for $q=2,3$, cf \cite[\S10]{BV}. In fact, we can prove that
for all $q$. To do it, we will prove some auxiliary propositions.

\begin{thm}\label{thmcl1}
If $O={\rm Aut}(X_I)C_{F_J}$ then for each $\bm{p}\in X_I(\mathbb{F}_{q^2})$ there are $C\in O$ and $L\in\mathcal{L}(X_I)$ such that $\bm{p}\in C\cap L$.
\end{thm}
\begin{proof}
Since $C_{F_J}$ is defined over $\mathbb{F}_{q^2}$ it passes a point in $X_I(\mathbb{F}_{q^2})$. Therefore by the transitivity of the ${\rm Aut}(X_I)$-action on $X_I(\mathbb{F}_{q^2})$, for each $\bm{p}\in X_I(\mathbb{F}_{q^2})$ there is $C\in O$ such that $\bm{p}\in C$. Since ${\rm Aut}(X_I)$ acts transitively on $\mathcal{R}(X_I)$, the proof completes in the same way by finding a line in $\mathcal{L}(X_I)$ defined over $\mathbb{F}_{q^2}$.
Indeed, let $L_G$ be the line defined by the image of the map
$
\mathbb{P}^1(k)\ni\lt(s,t)\mapsto G\,\lt(s,t)\in\mathbb{P}^3(k),
$
where 
$$
G=\pmat{1&0\\0&1\\\rho&0\\0&\rho}
$$
with $\rho^{q+1}=-1$. 
Then $L_G\in\mathcal{L}(X_I)$ and $\rho\in\mathbb{F}_{q^2}$, and so the lemma has been proven.
\end{proof}
By Theorem \ref{thmcl1} and the transitivity of ${\rm Aut}(X_I)$-action on $\mathcal{R}(X_I)$, we immediately have the following.
\begin{cor}\label{corcl}
Let $O$ be as in Theorem \ref{thmcl1}. Then there is $L\in\mathcal{L}(X_I)$ such that $C_{F_J}\cap L\neq\emptyset$.
\end{cor}
The following lemma is a key ingredient.
\begin{lem}\label{lemcl}
If $O={\rm Aut}(X_I)C_{F_J}$ then
$|C(\mathbb{F}_{q^2})\cap L(\mathbb{F}_{q^2})|\le1$ for any $C\in O$ and $L\in\mathcal{L}(X_I)$.
\end{lem}
\begin{proof}
By the transitivity on $\mathcal{R}(X_I)$ and Corollary \ref{corcl}, it suffices to prove that $|C_{F_J}(\mathbb{F}_{q^2})\cap L(\mathbb{F}_{q^2})|=1$ for $L\in\mathcal{L}(X_I)$ such that $C_{F_J}\cap L\neq\emptyset$.
Assume that $|C_{F_J}(\mathbb{F}_{q^2})\cap L(\mathbb{F}_{q^2})|\ge2$. Put $L=L_G$, where $G\in\mathcal{M}_{4,2}(\mathbb{F}_{q^2})$ is of rank $2$. Then, since $\mathbb{P}^1(k)\rightarrow C_{F_J}$ is injective, there are $\lt(a_1,b_1),\lt(a_2,b_2)\in\mathbb{P}^1(k)$ with  $\lt(a_1,b_1)\neq\lt(a_2,b_2)$ such that 
\begin{equation}\label{lemcl1}
F_J\pmat{{a_1}^{q+1}&{a_2}^{q+1}\\{a_1}^qb_1&{a_2}^qb_2\\a_1{b_1}^q&a_2{b_2}^q\\{b_1}^{q+1}&{b_2}^{q+1}}=G.
\end{equation}
Since $L_G\subset X_I$ we have
$$
\pmat{{a_1}^{q+1} & {a_1}^q b_1 & a_1{ b_1}^q & {b_1}^{q+1} \\ {a_2}^{q+1} & {a_2}^q b_2 & a_2 {b_2}^q & {b_2}^{q+1}}\lt F_J{F_J}^{(q)}\pmat{{a_1}^{q+1}&{a_2}^{q+1}\\{a_1}^qb_1&{a_2}^qb_2\\a_1{b_1}^q&a_2{b_2}^q\\{b_1}^{q+1}&{b_2}^{q+1}}^{(q)}=\pmat{0&0\\0&0},
$$
or equivalently,
$$
\begin{cases}
-({a_1}^qb_2)^{q+1}+{a_1}^{q^2}{b_1}^q{a_2}^qb_2+{a_1}^q{b_1}^{q^2}a_2{b_2}^q-({b_1}^qa_2)^{q+1}=0,\\
-(b_1{a_2}^q)^{q+1}+{a_1}^{q}{b_1}{a_2}^{q^2}{b_2}^q+{a_1}{b_1}^{q}{a_2}^q{b_2}^{q^2}-(a_1{b_2}^q)^{q+1}=0.
\end{cases}
$$
From these equations, we have 
\begin{center}
$a_1=0\iff a_2=0$ and $b_1=0\iff b_2=0$, 
\end{center}
and so $a_1b_1a_2b_2\neq0$. Therefore without loss of generality, one may put $b_1=b_2=1$ and $a_2=ca_1\neq0$ for $c\neq0,1$. Then $a_1,a_2\in{\mathbb{F}_{q^2}}^\times$ by \eqref{lemcl1}, and furthermore the equations are simplified to
\begin{equation*}
\begin{cases}
(c^{q^2}-c^{q^2+q}){a_1}^{q^2+q}+(c^q-1){a_1}^{q+1}=0,\\
(c^q-1){a_1}^{q^2+q}+(c-c^{q+1}){a_1}^{q+1}=0.
\end{cases}
\end{equation*}
Hence
we have
\begin{center}
$(c^{q^2}{a_1}^{q^2-1}-1)(c^q-1)=0$ and $(c-{a_1}^{q^2-1})(c^{q}-1)=0$,
\end{center}
and thus $c={a_1}^{q^2-1}$ and $c^{q^2}{a_1}^{q^2-1}=1$. However ${a_1}^{q^2-1}=1$ since $a_1\in{\mathbb{F}_{q^2}}^\times$, and so $c=1$
a contradiction, thus the proof has completed.
\end{proof}
It is well-known that $X_I(\mathbb{F}_{q^2})$ and $\mathcal{L}(X_I)$ form ${\rm GQ}(q^2,q)$, its the points graphs whose the parameters are represented by certain polynomials in $q^2,q$, cf. \cite{PT}. Lemma \ref{lemcl} implies that $E$ of Construction \ref{con1} is equal to the set of two vertices not joined by a line in $\mathcal{L}(X_I)$.
Therefore we have the following.
\begin{thm}
Let $O={\rm Aut}(X_I)C_{F_J}$. Then the graphs $(V,E)$ obtained by Construction \ref{con1} are the complement graphs of the point graphs of ${\rm GQ}(q^2,q)$ for all $q$. Their parameters can be written as
$$
(v,k,\lambda,\mu)=\left((q^3+1)(q^2+1),q^5,q(q-1)(q^3+q^2-1),q^3(q^2-1)\right).
$$
\end{thm}

\section{Association schemes from intersection of rational curves}
\begin{const}\label{con2}
Let $O$ be as in Construction 1 and let $\{m_1,\dots,m_d\}$ the set of the intersection numbers $I(C_1,C_2)$ for all $C_1,C_2\in O$ with $C_1\neq C_2$.
Define relations $R_i$ on $O$ for $1\le i\le d$ such that 
\begin{center}
$C_1\overset{R_i}{\sim}C_2$ $\overset{\mathrm{def.}}{\iff}$ $I(C_1,C_2)=m_i$,
\end{center}
together with the identity relation $R_0$. Then $\left(O,\{R_i\}_{i=0}^d\right)$ is a $d$-classes symmetric scheme of order $|O|$, where if $O={\rm Aut}(X_I)C_{F_J}$ or $q\ge3$ then $|O|=
q^4(q^3+1)(q^2-1)$.
\end{const}

\begin{ex}\label{con2ex1}
Let $q=2$. Then $\left(O,\{R_i\}_{i=0}^d\right)$ is of $d=5$ and order $432$, where $\{m_1,\dots,m_5\}=\{1,2,3,4,5\}$. The 1st eigenmatrix $P$ and the 2nd eigenmatrix $Q$ are
$$
\setlength{\arraycolsep}{2pt}
P=\pmat{
1&5&120&180&120&6\\
1&5&60&0&-60&-6\\
1&5& -24&36&-24&6\\
1&5&12&-36&12&6\\
1&5&-12&0&12&-6\\
1&-1&0&0&0&0
}
,
\ \ 
Q=\pmat{
1&6&15&20&30&360\\
1&6&15&20&30&-72\\
1&3&-3&2&-3&0\\
1&0&3&-4&0&0\\
1&-3&-3&2&3&0\\
1&-6&15&20&-30&0
},
$$
and the intersection matrices $L_i$ and the dual intersection matrices $L_i^*$ are as follows.
\begin{table}[H]
  \setlength{\arraycolsep}{1.5pt}
\renewcommand{\arraystretch}{0.8}
  \begin{center}
      \begin{tabular}{ccc}
      \\
           $L_0$&$L_2$&$L_4$\\
\hline
\\
$\pmat{1&0&0&0&0&0\\0&1&0&0&0&0\\0&0&1&0&0&0\\0&0&0&1&0&0\\0&0&0&0&1&0\\0&0&0&0&0&1}$&$\pmat{0 & 0 & 120 & 0 & 0 & 0 \\
0 & 0 & 120 & 0 & 0 & 0 \\
1 & 5 & 54 & 54 & 6 & 0 \\
0 & 0 & 36 & 48 & 36 & 0 \\
0 & 0 & 6 & 54 & 54 & 6 \\
0 & 0 & 0 & 0 & 120 & 0}$&$
\pmat{0 & 0 & 0 & 0 & 120 & 0 \\
0 & 0 & 0 & 0 & 120 & 0 \\
0 & 0 & 6 & 54 & 54 & 6 \\
0 & 0 & 36 & 48 & 36 & 0 \\
1 & 5 & 54 & 54 & 6 & 0 \\
0 & 0 & 120 & 0 & 0 & 0}$\\
\\
$L_1$&$L_3$&$L_5$\\
\hline
\\
$\pmat{0 & 5 & 0 & 0 & 0 & 0 \\
1 & 4 & 0 & 0 & 0 & 0 \\
0 & 0 & 5 & 0 & 0 & 0 \\
0 & 0 & 0 & 5 & 0 & 0 \\
0 & 0 & 0 & 0 & 5 & 0 \\
0 & 0 & 0 & 0 & 0 & 5}$
&$\pmat{0 & 0 & 0 & 180 & 0 & 0 \\
0 & 0 & 0 & 180 & 0 & 0 \\
0 & 0 & 54 & 72 & 54 & 0 \\
1 & 5 & 48 & 72 & 48 & 6 \\
0 & 0 & 54 & 72 & 54 & 0 \\
0 & 0 & 0 & 180 & 0 & 0}$
&$\pmat{0 & 0 & 0 & 0 & 0 & 6 \\
0 & 0 & 0 & 0 & 0 & 6 \\
0 & 0 & 0 & 0 & 6 & 0 \\
0 & 0 & 0 & 6 & 0 & 0 \\
0 & 0 & 6 & 0 & 0 & 0 \\
1 & 5 & 0 & 0 & 0 & 0}$
      \end{tabular}
  \end{center}
\end{table}

\begin{table}[H]
  \setlength{\arraycolsep}{1.7pt}
\renewcommand{\arraystretch}{1.0}
  \begin{center}
      \begin{tabular}{ccc}
      \\
           $L^*_0$&$L^*_2$&$L^*_4$\\ 
\hline
\\
$\pmat{1&0&0&0&0&0\\0&1&0&0&0&0\\0&0&1&0&0&0\\0&0&0&1&0&0\\0&0&0&0&1&0\\0&0&0&0&0&1}$&$\pmat{0 & 0 & 15 & 0 & 0 & 0 \\
0 & 0 & 0 & 0 & 15 & 0 \\
1 & 0 & 6 & 8 & 0 & 0 \\
0 & 0 & 6 & 9 & 0 & 0 \\
0 & 3 & 0 & 0 & 12 & 0 \\
0 & 0 & 0 & 0 & 0 & 15}$&$\pmat{0 & 0 & 0 & 0 & 30 & 0 \\
0 & 0 & 15 & 15 & 0 & 0 \\
0 & 6 & 0 & 0 & 24 & 0 \\
0 & \tfrac{9}{2} & 0 & 0 & \tfrac{51}{2} & 0 \\
1 & 0 & 12 & 17 & 0 & 0 \\
0 & 0 & 0 & 0 & 0 & 30}$\\
\\
$L^*_1$&$L^*_3$&$L^*_5$\\
\hline
\\
$\pmat{0 & 6 & 0 & 0 & 0 & 0 \\
1 & 0 & 0 & 5 & 0 & 0 \\
0 & 0 & 0 & 0 & 6 & 0 \\
0 & \tfrac{3}{2} & 0 & 0 & \tfrac{9}{2} & 0 \\
0 & 0 & 3 & 3 & 0 & 0 \\
0 & 0 & 0 & 0 & 0 & 6}$
&$\pmat{0 & 0 & 0 & 20 & 0 & 0 \\
0 & 5 & 0 & 0 & 15 & 0 \\
0 & 0 & 8 & 12 & 0 & 0 \\
1 & 0 & 9 & 10 & 0 & 0 \\
0 & 3 & 0 & 0 & 17 & 0 \\
0 & 0 & 0 & 0 & 0 & 20}$
&$\pmat{0 & 0 & 0 & 0 & 0 & 360 \\
0 & 0 & 0 & 0 & 0 & 360 \\
0 & 0 & 0 & 0 & 0 & 360 \\
0 & 0 & 0 & 0 & 0 & 360 \\
0 & 0 & 0 & 0 & 0 & 360 \\
1 & 6 & 15 & 20 & 30 & 288}$
      \end{tabular}
  \end{center}
\end{table}
\end{ex}
\begin{ex}\label{con2ex2}
Let $q=3$. Then $\left(O,\{R_i\}_{i=0}^d\right)$ is of $d=10$ and order $18144$, where $\{m_1,\dots,m_{10}\}=\{1,2,3,4,5,6,7,8,10,20\}$. 
\end{ex}
In the above examples for $q=2,3$, note that none of the numbers $m_i$ is zero.
In fact, the following theorem guarantees that it is true for any $q$.
\begin{thm}\label{lem5}
If $O={\rm Aut}(X_I)C_{F_J}$,
two distinct curves of $O$ always intersect.
\end{thm}
\begin{proof}
Since all curves of $O$ are defined over $\mathbb{F}_{q^2}$, there is $\bm{v}\in C\cap X_I(\mathbb{F}_{q^2})$ for each $C\in O$. Then $C$ and $g\,C$ for $g\in{\rm Stab}(\bm{v})$ intersect at $\bm{v}$, and furthermore $g$ can be taken so that $C\neq g\,C$, that is, $g\notin{\rm Stab}(C)$ because
$$
|{\rm Stab}(C)|=q^2(q^4-1)<q^6(q^2-1)^2
=|{\rm Stab}(\bm{v})|.
$$
Further the set consisting of $g\,C$ for all such $g$ is equal to $O$, since
$$
|{\rm Stab}(\bm{v})|-|{\rm Stab}(C)|-|O|>0.
$$
Therefore the lemma has been proven.
\end{proof}
\begin{cor}
Let $O$ be as in Theorem \ref{lem5} and let
$$
E=\left\{(C_1,C_2)\in O\times O\ \middle\vert\ C_1\neq C_2,\ C_1\cap C_2\neq\emptyset \right\}.
$$
Then $(O,E)$ is a complete graph with vertex number $v=|O|$. 
\end{cor}
\begin{rem}
It is easy to show that there are two distinct lines on $X_I$ not intersecting each other. Indeed, as an example of such two curves we can give their defining equations as follows:
$$
L_1:\ \begin{cases}x_0=\rho x_1\\
x_2=\rho x_3
\end{cases}
,\ \ 
L_2:\ \begin{cases}x_0={\rho}'x_1\\
x_2={\rho}'x_3
\end{cases}
,
$$
where $\rho,{\rho}'\in\mathbb{F}_{q^2}$ with $\rho^{q+1}={{\rho}'}^{q+1}=-1$ and $\rho\neq{\rho}'$.
Thus, Theorem \ref{lem5} stands in striking contrast to the case of the lines.
\end{rem}
We also notice in Examples \ref{con2ex1},\ref{con2ex2} that
$d=|\mathbb{P}^1(\mathbb{F}_{q^2})|$. This observation leads to the following conjecture:
\begin{conj}\label{nin}
If $O={\rm Aut}(X_I)C_{F_J}$ then
$$d=|\mathbb{P}^1(\mathbb{F}_{q^2})|\ {for}\ {all}\ q,$$ 
that is, $\left(O,\{R_i\}_{i=0}^d\right)$ is of $d=(q^2+1)$-classes. 
\end{conj}
To clarify general properties of $\left(O,\{R_i\}_{i=0}^d\right)$ such as this conjecture it seems to need more detailed investigation for the intersection of curves of $O$. It is left for another occasion.

\section{Schurian schemes for rational points, lines and curves}
\begin{const}
\label{con3}
Let $O$ be as in Construction \ref{con1}. Let $S=X_I(\mathbb{F}_{q^2})$, $\mathcal{L}(X_I)$ or $O$ and let
$\left\{R_0,R_1,\dots,R_d\right\}$ the partition of $S\times S$ obtained by the diagonal action of ${\rm Aut}(X_I)$, where $R_0=\left\{(gx,gx)\ \middle\vert\ x\in S,\ g\in{\rm Aut}(X_I)\right\}$. Then $(S,\{R_i\}_{i=0}^d)$ is a $d$-classes scheme of order $|S|$. Especially the case where $S=X_I(\mathbb{F}_{q^2})$ or $\mathcal{L}(X_I)$ is the $2$-classes symmetric scheme corresponding to the point graph or the line graph of ${\rm GQ}(q^2,q)$, respectively.
\end{const}

\begin{ex}
Let $q=2$.
Then $(X_I(\mathbb{F}_{q^2}),\{R_i\}_{i=0}^d)$ is of $d=2$ and order $45$. The eigenmatrices are
$$
\setlength{\arraycolsep}{2pt}
P=\pmat{
1&32&12\\
1&2&-3\\
1&-4&3
}
,
\ \ 
Q=\pmat{
1&24&20\\
1&\frac{3}{2}&-\frac{5}{2}\\
1&-6&5
},
$$
and the intersection matrices and the dual intersection matrices are as follows.
\begin{table}[H]
  \setlength{\arraycolsep}{1.7pt}
\renewcommand{\arraystretch}{1.0}
  \begin{center}
      \begin{tabular}{ccc}
      \\
           $L_0$&$L_1$&$L_2$\\
\hline
\\
$\pmat{1&0&0\\0&1&0\\0&0&1}$&$\pmat{0 & 32 & 0 \\
1 & 22 & 9 \\
0 & 24 & 8}$&$\pmat{0 & 0 & 12 \\
0 & 9 & 3 \\
1 & 8 & 3}$\\
\\
$L^*_0$&$L^*_1$&$L^*_2$\\
\hline
\\
$\pmat{1&0&0\\0&1&0\\0&0&1}$
&$\pmat{0 & 24 & 0 \\
1 & \tfrac{21}{2} & \tfrac{25}{2} \\
0 & 15 & 9}$
&$\pmat{0 & 0 & 20 \\
0 & \tfrac{25}{2} & \tfrac{15}{2} \\
1 & 9 & 10}$
      \end{tabular}
  \end{center}
\end{table}
\end{ex}

\begin{ex}
Let $q=2$.
Then $(\mathcal{L}(X_I),\{R_i\}_{i=0}^d)$ is of $d=2$ and order $27$. The eigenmatrices are
$$
\setlength{\arraycolsep}{2pt}
P=\pmat{
1 & 10 & 16 \\ 1 & 1 & -2 \\ 1 & -5 & 4
}
,
\ \ 
Q=\pmat{
1 & 20 & 6 \\ 1 & 2 & -3 \\ 1 & -\frac{5}{2} & \frac{3}{2}
},
$$
and the intersection matrices and the dual intersection matrices are as follows.
\begin{table}[H]
  \setlength{\arraycolsep}{1.7pt}
\renewcommand{\arraystretch}{1.0}
  \begin{center}
      \begin{tabular}{ccc}
      \\
           $L_0$&$L_1$&$L_2$\\
\hline
\\
$\pmat{1&0&0\\0&1&0\\0&0&1}$&$\pmat{0 & 10 & 0 \\ 1 & 1 & 8 \\ 0 & 5 & 5}$&$\pmat{0 & 0 & 16 \\ 0 & 8 & 8 \\ 1 & 5 & 10}$\\
\\
$L^*_0$&$L^*_1$&$L^*_2$\\
\hline
\\
$\pmat{1&0&0\\0&1&0\\0&0&1}$
&$\pmat{0 & 1 & 0 \\ 20 & \frac{29}{2} & 15 \\ 0 & \frac{9}{2} & 5}$
&$\pmat{0 & 0 & 1 \\ 0 & \frac{9}{2} & 5 \\ 6 & \frac{3}{2} & 0}$
      \end{tabular}
  \end{center}
\end{table}
\end{ex}

\begin{ex}\label{SCS}
Let $q=2$. Then $(O,\{R_i\}_{i=0}^d)$ is a $d=19$-classes and order $432$ scheme which is noncommutative. 

Generally, the character table $P$ of a $d$-classes association scheme $\mathcal{S}$ with vertices $V$ is defined as the matrix in $\mathcal{M}_{r+1,d+1}(\mathbb{C})$ consisting of
the trace values ${\rm Tr}(\varphi_i(A_j))$ of the irreducible representations $\{\varphi_i\}_{i=0}^r$ of the algebra $\mathcal{A}=\mathbb{C}[A_0,\dots,A_d]$ generated by the adjacent matrices $A_0,\dots,A_d$ of $\mathcal{S}$, and the dual character table $Q$ is defined as the matrix in $\mathcal{M}_{d+1,r+1}(\mathbb{C})$ such that 
$$
E_j=|V|^{-1}\sum_{i=0}^dQ_{i,j}A_i
$$
for $0\le j\le r$, where $E_0,\dots,E_r$ are the commutative and primitive idempotents which span the center of $\mathcal{A}$. Then $PQ=|V|D$, where $D$ is the diagonal matrix in $\mathcal{M}_{r+1}(\mathbb{Z})$ with the $i$-th diagonal entry is $P_{i,0}$
for every $i$. 
Note that $X\in\mathcal{M}_{d+1,r+1}(\mathbb{C})$ such that $PX=|V|D$ is not unique.
When ${\rm rank}(\varphi_i)=1$ for all $i$, the scheme is commutative and its the character tables are usually called eigenmatrices.

In this example, $P$ and $Q$ are actually computed as in Table \ref{cts}.
\begin{table}[h]
\caption{The Character tables $P,Q$ for the scheme of Example \ref{SCS}}\label{cts}
\[
\scriptsize
\setlength{\tabcolsep}{0.7pt}
\renewcommand{\arraystretch}{0.8}
\begin{tabular}{c||c|c|c|c|c|c|c|c|c|c|c|c|c|c|c|c|c|c|c|c}
$P$&$A_0$&$A_1$&$A_2$&$A_3$&$A_4$&$A_5$&$A_6$&$A_7$&$A_8$&$A_9$&$A_{10}$&$A_{11}$&$A_{12}$&$A_{13}$&$A_{14}$&$A_{15}$&$A_{16}$&$A_{17}$&$A_{18}$&$A_{19}$\\
\hline\hline
$\varphi_0$&\(1\) & \(1\) & \(20\) & \(20\) & \(5\) & \(10\) & \(20\) & \(20\) & \(10\) & \(5\) & \(10\) & \(10\) & \(30\) & \(30\) & \(30\) & \(30\) & \(30\) & \(30\) & \(60\) & \(60\) \\
$\varphi_1$&\(1\) & \(1\) & \(-2\xi_{21}\) & \(-2\xi_{21}\) & \(-1\) & \(\overline{\xi_{12}}\) & \(-2\overline{\xi_{21}}\) & \(-2\overline{\xi_{21}}\) & \(\xi_{12}\) & \(-1\) & \(\overline{\xi_{12}}\) & \(\xi_{12}\) & \(3\) & \(0\) & \(3\) & \(0\) & \(3\) & \(3\) & \(0\) & \(0\) \\
$\varphi_2$&\(1\) & \(-1\) & \(2\) & \(-2\) & \(1\) & \(-1\) & \(-2\) & \(2\) & \(1\) & \(-1\) & \(1\) & \(-1\) & \(-3\) & \(-6\) & \(3\) & \(6\) & \(3\) & \(-3\) & \(0\) & \(0\) \\
$\varphi_3$&\(1\) & \(-1\) & \(4\overline{\xi_{21}}\) & \(-4\overline{\xi_{21}}\) & \(1\) & \(2\xi_{12}\) & \(-4\xi_{21}\) & \(4\xi_{21}\) & \(-2\overline{\xi_{12}}\) & \(-1\) & \(-2\xi_{12}\) & \(2\overline{\xi_{12}}\) & \(-6\) & \(18\) & \(6\) & \(-18\) & \(6\) & \(-6\) & \(12\xi\) & \(12\overline{\xi}\) \\
$\varphi_4$&\(1\) & \(-1\) & \(4\xi_{21}\) & \(-4\xi_{21}\) & \(1\) & \(2\overline{\xi_{12}}\) & \(-4\overline{\xi_{21}}\) & \(4\overline{\xi_{21}}\) & \(-2\xi_{12}\) & \(-1\) & \(-2\overline{\xi_{12}}\) & \(2\xi_{12}\) & \(-6\) & \(18\) & \(6\) & \(-18\) & \(6\) & \(-6\) & \(12\overline{\xi}\) & \(12\xi\) \\
$\varphi_5$&\(1\) & \(1\) & \(2\) & \(2\) & \(-1\) & \(-1\) & \(2\) & \(2\) & \(-1\) & \(-1\) & \(-1\) & \(-1\) & \(-1\) & \(-4\) & \(-1\) & \(-4\) & \(-1\) & \(-1\) & \(4\) & \(4\) \\
$\varphi_6$&\(1\) & \(-1\) & \(-10\) & \(10\) & \(-5\) & \(5\) & \(10\) & \(-10\) & \(-5\) & \(5\) & \(-5\) & \(5\) & \(-15\) & \(0\) & \(15\) & \(0\) & \(15\) & \(-15\) & \(0\) & \(0\) \\
$\varphi_7$&\(1\) & \(1\) & \(2\) & \(2\) & \(-1\) & \(4\) & \(2\) & \(2\) & \(4\) & \(-1\) & \(4\) & \(4\) & \(-6\) & \(6\) & \(-6\) & \(6\) & \(-6\) & \(-6\) & \(-6\) & \(-6\) \\
$\varphi_8$&\(1\) & \(-1\) & \(-\xi_{11}\) & \(\xi_{11}\) & \(1\) & \(\xi_{21}\) & \(\overline{\xi_{11}}\) & \(-\overline{\xi_{11}}\) & \(-\overline{\xi_{21}}\) & \(-1\) & \(-\xi_{21}\) & \(\overline{\xi_{21}}\) & \(3\) & \(0\) & \(-3\) & \(0\) & \(-3\) & \(3\) & \(6\overline{\xi}\) & \(6\xi\) \\
$\varphi_9$&\(1\) & \(-1\) & \(-\overline{\xi_{11}}\) & \(\overline{\xi_{11}}\) & \(1\) & \(\overline{\xi_{21}}\) & \(\xi_{11}\) & \(-\xi_{11}\) & \(-\xi_{21}\) & \(-1\) & \(-\overline{\xi_{21}}\) & \(\xi_{21}\) & \(3\) & \(0\) & \(-3\) & \(0\) & \(-3\) & \(3\) & \(6\xi\) & \(6\overline{\xi}\) \\
$\varphi_{10}$&\(1\) & \(1\) & \(2\) & \(2\) & \(5\) & \(1\) & \(2\) & \(2\) & \(1\) & \(5\) & \(1\) & \(1\) & \(3\) & \(-6\) & \(3\) & \(-6\) & \(3\) & \(3\) & \(-12\) & \(-12\) \\
$\varphi_{11}$&\(1\) & \(1\) & \(-2\overline{\xi_{21}}\) & \(-2\overline{\xi_{21}}\) & \(-1\) & \(\xi_{12}\) & \(-2\xi_{21}\) & \(-2\xi_{21}\) & \(\overline{\xi_{12}}\) & \(-1\) & \(\xi_{12}\) & \(\overline{\xi_{12}}\) & \(3\) & \(0\) & \(3\) & \(0\) & \(3\) & \(3\) & \(0\) & \(0\) \\
$\varphi_{12}$&\(4\) & \(4\) & \(-4\) & \(-4\) & \(8\) & \(-14\) & \(-4\) & \(-4\) & \(-14\) & \(8\) & \(-14\) & \(-14\) & \(-6\) & \(36\) & \(-6\) & \(36\) & \(-6\) & \(-6\) & \(0\) & \(0\) \\
$\varphi_{13}$&\(4\) & \(-4\) & \(-4\) & \(4\) & \(-8\) & \(-10\) & \(4\) & \(-4\) & \(10\) & \(8\) & \(10\) & \(-10\) & \(6\) & \(12\) & \(-6\) & \(-12\) & \(-6\) & \(6\) & \(0\) & \(0\)
\end{tabular}
\]
\\



\[
\scriptsize
\setlength{\tabcolsep}{0.7pt}
\renewcommand{\arraystretch}{0.8}
\begin{tabular}{c||c|c|c|c|c|c|c|c|c|c|c|c|c|c}
$Q$ & $\varphi_0$ & $\varphi_1$ & $\varphi_2$ & $\varphi_3$ & $\varphi_4$ & $\varphi_5$ & $\varphi_6$ & $\varphi_7$ & $\varphi_8$ & $\varphi_9$ & $\varphi_{10}$ & $\varphi_{11}$ & $\varphi_{12}$ & $\varphi_{13}$ \\ \hline\hline
$A_0$  & $1$ & $30$ & $60$ & $5$ & $5$ & $81$ & $6$ & $24$ & $40$ & $40$ & $20$ & $30$ & $30$ & $60$ \\
$A_1$  & $1$ & $30$ & $-60$ & $-5$ & $-5$ & $81$ & $-6$ & $24$ & $-40$ & $-40$ & $20$ & $30$ & $30$ & $-60$ \\
$A_2$  & $1$ & $-3\,\overline{\xi_{21}}$ & $6$ & $\xi_{21}$ & $\overline{\xi_{21}}$ & $81/10$ & $-3$ & $12/5$ & $-2\,\overline{\xi_{11}}$ & $-2\,\xi_{11}$ & $2$ & $-3\,\xi_{21}$ & $-3/2$ & $-3$ \\
$A_3$  & $1$ & $-3\,\overline{\xi_{21}}$ & $-6$ & $-\xi_{21}$ & $-\overline{\xi_{21}}$ & $81/10$ & $3$ & $12/5$ & $2\,\overline{\xi_{11}}$ & $2\,\xi_{11}$ & $2$ & $-3\,\xi_{21}$ & $-3/2$ & $3$ \\
$A_4$  & $1$ & $-6$ & $12$ & $1$ & $1$ & $-81/5$ & $-6$ & $-24/5$ & $8$ & $8$ & $20$ & $-6$ & $12$ & $-24$ \\
$A_5$  & $1$ & $3\,\xi_{12}$ & $-6$ & $\overline{\xi_{12}}$ & $\xi_{12}$ & $-81/10$ & $3$ & $48/5$ & $4\,\overline{\xi_{21}}$ & $4\,\xi_{21}$ & $2$ & $3\,\overline{\xi_{12}}$ & $-21/2$ & $-15$ \\
$A_6$  & $1$ & $-3\,\xi_{21}$ & $-6$ & $-\overline{\xi_{21}}$ & $-\xi_{21}$ & $81/10$ & $3$ & $12/5$ & $2\,\xi_{11}$ & $2\,\overline{\xi_{11}}$ & $2$ & $-3\,\overline{\xi_{21}}$ & $-3/2$ & $3$ \\
$A_7$  & $1$ & $-3\,\xi_{21}$ & $6$ & $\overline{\xi_{21}}$ & $\xi_{21}$ & $81/10$ & $-3$ & $12/5$ & $-2\,\xi_{11}$ & $-2\,\overline{\xi_{11}}$ & $2$ & $-3\,\overline{\xi_{21}}$ & $-3/2$ & $-3$ \\
$A_8$  & $1$ & $3\,\overline{\xi_{12}}$ & $6$ & $-\xi_{12}$ & $-\overline{\xi_{12}}$ & $-81/10$ & $-3$ & $48/5$ & $-4\,\xi_{21}$ & $-4\,\overline{\xi_{21}}$ & $2$ & $3\,\xi_{12}$ & $-21/2$ & $15$ \\
$A_9$  & $1$ & $-6$ & $-12$ & $-1$ & $-1$ & $-81/5$ & $6$ & $-24/5$ & $-8$ & $-8$ & $20$ & $-6$ & $12$ & $24$ \\
$A_{10}$& $1$ & $3\,\xi_{12}$ & $6$ & $-\overline{\xi_{12}}$ & $-\xi_{12}$ & $-81/10$ & $-3$ & $48/5$ & $-4\,\overline{\xi_{21}}$ & $-4\,\xi_{21}$ & $2$ & $3\,\overline{\xi_{12}}$ & $-21/2$ & $15$ \\
$A_{11}$& $1$ & $3\,\overline{\xi_{12}}$ & $-6$ & $\xi_{12}$ & $\overline{\xi_{12}}$ & $-81/10$ & $3$ & $48/5$ & $4\,\xi_{21}$ & $4\,\overline{\xi_{21}}$ & $2$ & $3\,\xi_{12}$ & $-21/2$ & $-15$ \\
$A_{12}$& $1$ & $3$ & $-6$ & $-1$ & $-1$ & $-27/10$ & $-3$ & $-24/5$ & $4$ & $4$ & $2$ & $3$ & $-3/2$ & $3$ \\
$A_{13}$& $1$ & $0$ & $-12$ & $3$ & $3$ & $-54/5$ & $0$ & $24/5$ & $0$ & $0$ & $-4$ & $0$ & $9$ & $6$ \\
$A_{14}$& $1$ & $3$ & $6$ & $1$ & $1$ & $-27/10$ & $3$ & $-24/5$ & $-4$ & $-4$ & $2$ & $3$ & $-3/2$ & $-3$ \\
$A_{15}$& $1$ & $0$ & $12$ & $-3$ & $-3$ & $-54/5$ & $0$ & $24/5$ & $0$ & $0$ & $-4$ & $0$ & $9$ & $-6$ \\
$A_{16}$& $1$ & $3$ & $6$ & $1$ & $1$ & $-27/10$ & $3$ & $-24/5$ & $-4$ & $-4$ & $2$ & $3$ & $-3/2$ & $-3$ \\
$A_{17}$& $1$ & $3$ & $-6$ & $-1$ & $-1$ & $-27/10$ & $-3$ & $-24/5$ & $4$ & $4$ & $2$ & $3$ & $-3/2$ & $3$ \\
$A_{18}$& $1$ & $0$ & $0$ & $-\xi$ & $\xi$ & $27/5$ & $0$ & $-12/5$ & $4\xi$ & $-4\xi$ & $-4$ & $0$ & $0$ & $0$ \\
$A_{19}$& $1$ & $0$ & $0$ & $\xi$ & $-\xi$ & $27/5$ & $0$ & $-12/5$ & $-4\xi$ & $4\xi$ & $-4$ & $0$ & $0$ & $0$ \\
\end{tabular}
\]
\end{table}
In the table, we put
\begin{center}
$\xi=\sqrt{3}\,i$, $\xi_{11}=1+\sqrt{3}\,i$, $\xi_{12}=1+2\sqrt{3}\,i$, $\xi_{21}=2+\sqrt{3}\,i$ for $i^2=-1$.
\end{center}
Besides, $\overline{\xi}$ denotes the complex conjugate of $\xi$. As usual, we set $A_0=I$ and furthermore $\varphi_0(X):\bm{v}\mapsto\bm{v}X$ for $\bm{v}\in V_0$, where $V_0$ is the eigenspace of ${432}^{-1}J$ for which $J$ is the all-ones matrix of $\mathcal{M}_{432}(\mathbb{Z})$. From the table, the valency
$$
[k_0,\dots,k_{19}]=
[1, 1, 20, 20, 5, 10, 20, 20, 10, 5, 10, 10, 30, 30, 30, 30, 30, 30, 60, 60]
$$
and
$$
[{\rm rank}(E_0),\dots,{\rm rank}(E_{13})]=
[ 1, 30, 60, 5, 5, 81, 6, 24, 40, 40, 20, 30, 30, 60 ].
$$
Further, we have
$$
[{\rm rank}(\varphi_0),\dots,{\rm rank}(\varphi_{13})]=[1,1,1,1,1,1,1,1,1,1,1,1,2,2]
$$
because 
$$
\sum_{i=0}^{13}{\rm rank}(\varphi_i)^2={\rm dim}_{\mathbb{C}}(\mathcal{A})=d+1=20
$$ 
and $P_{i,0}>1$ for $i=12,13$. Hence the multiplicity
$$
[m_0,\dots,m_{13}]=
[1, 30, 60, 5, 5, 81, 6, 24, 40, 40, 20, 30, 15, 30 ].
$$
It should be noticed that all the entries of $P,Q$ are contained in the cyclotomic field of degree $3$.
\end{ex}
\begin{rem}
Since a projective transformation preserves the intersection multiplicity at each point, letting $\{m_1,\dots,m_l\}$ be 
the set of the intersection numbers of all two distinct curves of $O$ and $e$ the number of the intersection numbers distinguished by the intersection multiplicities of those curves, for $(O,\{R_i\}_{i=0}^d)$ of Construction 3 we have
$$
\min\left(d,\mathfrak{p}(m_1)+\dots+\mathfrak{p}(m_l)\right)\ge e,
$$
where $\mathfrak{p}$ denotes the partition function. For $q=2$, by Examples \ref{SCS} and \ref{con2ex1},
$$
d=19\ge\mathfrak{p}(1)+\mathfrak{p}(2)+\mathfrak{p}(3)+\mathfrak{p}(4)+\mathfrak{p}(5)=18\ge e,
$$
and for $q=3$, by Example \ref{con2ex2},
$$
\mathfrak{p}(1)+\dots+\mathfrak{p}(8)+\mathfrak{p}(10)+\mathfrak{p}(20)=735.
$$
If the above LHS is equals to $e$ then $d\ge735$. However the assumption is not immediately clear.
\end{rem}



\begin{thebibliography}{9999999}
\bibitem{BC} R. C. Bose and I. M. Chakravarti, ``Hermitian varieties in a finite projective space {\rm PG}$(N,q^2)$'', {\it Canad.\ J. Math.}\ 18:1161-1182, 1966.
\bibitem{BV} A. E. Brouwer and H. Van Maldeghem, ``Strongly Regular Graphs'', Cambridge University Press, 2022.
\bibitem{C} I. M. Chakravarti, ``Some properties and applications of Hermitian varieties in a finite projective space ${\rm PG}(N,q^2)$ in the construction of strongly regular graphs (two-class association schemes) and block designs'', {\it J. Combin.\ Theory} 11:268-283, 1971.
\bibitem{CDS} K. Coolsaet, J. Degraer and E. Spence, ``The Strongly Regular $(45,12,3,3)$ Graphs'', {\it Electronic journal of combinatorics} 13(1):R32, 2006. 
\bibitem{O} N. Ojiro, ``Rational curves on a smooth Hermitian surface'', {\it Hiroshima Math.\ J.} 49(1):161-173, 2019.
\bibitem{O2} N. Ojiro, ``Rational curves on a smooth Hermitian surface II'',  preprint arXiv:2003.13211.
\bibitem{PT} S. E. Payne and J. A. Thas. ``Finite Generalized Quadrangles'', 1984.
\bibitem{Se} B. Segre, ``Forme e geometrie hermitiane, con particolare riguardo al caso finito'', {\it Ann.\ Mat.\ Pura Appl.}\ (4) 70:1-201, 1965.
\bibitem{Sh1} I. Shimada, ``A note on rational normal curves totally tangent to a Hermitian variety'', {\it Des.\ Codes Cryptogr.}\ 69:299–303, 2013.
\bibitem{Sh2} I. Shimada, ``The graphs of Hoffman-Singleton, Higman-Sims and McLaughlin, and the
Hermitian curve of degree 6 in characteristic 5'', {\it Australas.\ J. Combin.}\ 59:161–181, 2014.
\end{thebibliography}
\end{document}